\documentclass{amsart}
\usepackage[utf8]{inputenc}
\usepackage{amsmath, amssymb, bbm}
\usepackage{amsthm}
\usepackage{mathtools}
\usepackage{bbm}
\usepackage{blkarray}
\usepackage[matrix,arrow,curve]{xy}
\usepackage{tikz}
\usepackage{graphicx}
\usepackage{color}
\usepackage{enumerate}
\definecolor{darkblue}{rgb}{0,0,0.6}
\usepackage[ocgcolorlinks,colorlinks=true, citecolor=darkblue, filecolor=darkblue, linkcolor=darkblue, urlcolor=darkblue]{hyperref}
\usepackage[capitalize,nameinlink]{cleveref}

\newcommand{\ignore}[1]{}
\renewcommand{\epsilon}{\varepsilon}
\renewcommand{\phi}{\varphi}

\newcommand{\Z}{\mathbbm{Z}}

\newcommand{\mcS}{\mathcal{S}}
\newcommand{\CP}{\mathbbm{CP}}
\newcommand{\RP}{\mathbbm{RP}}

\newcommand{\bbN}{\mathbbm{N}}

\newcommand{\ra}{\longrightarrow}

\newcommand{\wt}{\widetilde}

\newcommand{\Tor}{\mathrm{Tor}}

\DeclareMathOperator{\Sym}{Sym}

\DeclareMathOperator{\Out}{Out}
\DeclareMathOperator{\Hom}{Hom}
\DeclareMathOperator{\Her}{Herm}

\DeclareMathOperator{\id}{Id}
\DeclareMathOperator{\trace}{trace}
\DeclareMathOperator{\im}{im}

\DeclareMathOperator{\Id}{Id}

\DeclareMathOperator{\ev}{ev}

\DeclareMathOperator{\Tors}{Tors}
\DeclareMathOperator{\tr}{tr}
\DeclareMathOperator{\PD}{PD}

\numberwithin{equation}{section}
\newtheorem{thm}[equation]{Theorem}

\newtheorem{prop}[equation]{Proposition}

\newtheorem{cor}[equation]{Corollary}
\newtheorem{lemma}[equation]{Lemma}
\newtheorem*{thm*}{Theorem}
\theoremstyle{definition}

\newtheorem{defi}[equation]{Definition}

\newtheorem{rem}[equation]{Remark}

\newtheorem{question}[equation]{Question}

\crefname{thm}{Theorem}{Theorems}

\newcommand{\intf}{\lambda}
\newcommand{\DCM}[2]{\DeclareMathOperator{#1}{#2}}
\DCM{\Herm}{Herm}

\begin{document}

\title[$\CP^2$-stable classification of $4$-manifolds with finite $\pi_1$]{$\CP^2$-stable classification of $4$-manifolds with finite fundamental group}

\author[D.~Kasprowski]{Daniel Kasprowski}
\address{Rheinische Friedrich-Wilhelms-Universit\"at Bonn, Mathematisches Institut,\newline\indent Endenicher Allee 60, 53115 Bonn, Germany}
\email{kasprowski@uni-bonn.de}

\author[P.~Teichner]{Peter Teichner}
\address{Max Planck Institut f\"{u}r Mathematik, Vivatsgasse 7, 53111 Bonn, Germany}
\email{teichner@mpim-bonn.mpg.de}


\date{\today}

\begin{abstract}
We show that two closed, connected $4$-manifolds with finite fundamental groups are $\CP^2$-stably homeomorphic if and only if their quadratic $2$-types are stably isomorphic and their Kirby-Siebenmann invariants agree.
\end{abstract}
\maketitle
\section{Introduction}
Two $4$-manifolds are $\CP^2$-stably homeomorphic if they become homeomorphic after taking connected sum with finitely many copies of $\CP^2$, where we allow different numbers of copies for
the two manifolds. If the manifolds are orientable, we allow connected sums with both orientations of $\CP^2$. We will show that the $\CP^2$-stable classification for closed connected 4-manifolds with finite fundamental group is determined by their (stable) quadratic $2$-type. 

The quadratic $2$-type $Q_M$ consists of the Postnikov $2$-type of the 4-manifold $M$, determined by $\pi_1M, \pi_2M$ (as a $\pi_1M$-module) and the $k$-invariant in $H^3(\pi_1M;\pi_2M)$, together with the orientation character $w_1M: \pi_1M\to \{\pm 1\}$ and the equivariant intersection form $\lambda_M: \pi_2M \times \pi_2M \ra \Z[\pi_1M]$:
\[
Q_M:=(\pi_1M,w_1M,\pi_2M,k_M,\lambda_M).
\]
We say that the quadratic $2$-types of two manifolds are stably isomorphic if they become isomorphic after finitely many
stabilizations as follows:
\[
Q_M \mapsto Q_{M\#\pm\CP^2} \cong (\pi_1M,w_1M,\pi_2M \oplus \Z[\pi_1M],(k_M,0), \lambda_M \perp \langle \pm 1 \rangle ) 
\]
For a manifold $M$ with an isomorphism $\pi_1M \cong \pi$, we pick a continuous map $c\colon M\to B\pi$ inducing the isomorphism on $\pi_1$. Then $c$ is uniquely determined up to based homotopy by $\pi_1(c)$.
\begin{thm}
	\label{thm:main}
Let $M_1, M_2$ be closed, connected $4$-manifolds with finite fundamental group $\pi$, orientation character $w:\pi\to\{\pm 1\}$ and equal Kirby-Siebenmann invariants. Then the following are equivalent:
	\begin{enumerate}
		\item\label{it:main1} $c_*[M_1]=c_*[M_2]\in H_4(\pi;\Z^w)/\pm\Out(\pi)$;
		\item\label{it:main2} $M_1$ and $M_2$ are $\CP^2$-stably homeomorphic;
		\item\label{it:main3} the quadratic 2-types $Q_{M_1}$ and $Q_{M_2}$ are stably isomorphic.
	\end{enumerate}
\end{thm}
\begin{rem}
	Since stably homeomorphic smooth $4$-manifolds are stably diffeomorphic, \cref{thm:main} remains true in the smooth category (where Kirby-Siebenmann vanishes). In the topological case, one could also allow taking connected sums with Freedman's manifold $*\CP^2$. Then the Kirby-Siebenmann invariant can be changed and \cref{thm:main} has a version where one does not have to control it. 
\end{rem}
Since $\CP^2\#\CP^2\#(-\CP^2)\cong \CP^2\#(S^2\times S^2)$, two manifolds are $\CP^2$-stably homemorphic if they are even {\em stably homeomorphic}, i.e.\ homeomorphic after taking connected sum with finitely many copies of $S^2\times S^2$.

If follows from Kreck's modified surgery \cite{kreck} that \eqref{it:main1} and \eqref{it:main2} are equivalent for arbitrary fundamental groups, see also \cite[Thm~1.1]{KPT}. Moreover, \eqref{it:main2} implies \eqref{it:main3} by definition and hence the interesting implication for us is \eqref{it:main3}$\Rightarrow$\eqref{it:main1}. 

In \cite[Thm~A]{KPT} we showed that already the quadruple $(\pi_1,\pi_2,k_,w_1)$ determines the $\CP^2$-stable classification, for groups that have one end or are torsion-free. 

We also described two 4-manifolds of the form lens space times circle which both have  fundamental group $\Z/n\times\Z$ and $\pi_2=0$ but are not $\CP^2$-stably homeomorphic \cite[Ex. 1.4]{KPT}. Hence for infinite groups (with two ends and torsion), the implication \eqref{it:main3}$\Rightarrow$\eqref{it:main1} is false in general.

The following question remains open, compare \cite[Question~1.5]{KPT}, even though we show there that the answer is ``no'' if $H_4(\pi; \Z^w)$ is annihilated by 4 or 6.

\begin{question}
Are there finite groups such that the equivariant intersection form $\lambda$ is needed in \cref{thm:main} .
\end{question}

Our proof is based on the following results on {\em finite Poincar\'e complexes}. Let $X$ be a finite connected Poincar\'e 4-complex with finite fundamental group $\pi$, orientation character $w$ and chosen fundamental class $[X]\in H_4(X;\Z^w)\cong \Z$. If $B$ is a $3$-coconnected CW-complex then $X$ is \emph{$B$-polarized} if it is equipped with a $3$-equivalence $f\colon X\to B$. The set of $4$-dimensional $B$-polarized Poincar\'e complexes with orientation character $w:\pi_1B\to \{\pm 1\}$, up to homotopy equivalence over $B$, is denoted by $\mcS_4^{\PD}(B,w)$. 
\begin{thm}\label{thm:hk}
Assume that $\mcS_4^{\PD}(B,w)$ is non-empty and that $\pi:=\pi_1B$ is finite. Then there is an exact sequence
\[
0\to \Tors(\Z^w\otimes_{\Z\pi}\Gamma(\pi_2(B)))\to \mcS_4^{\PD}(B,w)\to \Z/|\pi| \times \Her(\pi_2(B)).
\]
\end{thm}
Hambleton-Kreck {\cite[Thm 1.1]{hambleton-kreck}} prove this result in the oriented case $w=0$ but we need the above generalization to the non-orientable case. The general case as well as \cref{thm:teichner} below were obtained by the second author in his PhD-thesis \cite{teichnerthesis} but have not yet been published, so we'll review them in this paper.

For a hermitian form $\lambda$ on $\pi_2 B$, let $\mcS_4^{\PD}(B,w,\lambda)$
be the subset of $\mcS_4^{\PD}(B,w)$ of $B$-polarized Poincar\'e complexes $(X,f)$ such that $\lambda$ is mapped to the intersection form $\lambda_{X}$ via $f^*$.

\begin{thm}
	\label{thm:teichner}
	Assume that $\mcS_4^{\PD}(B,w,\lambda)$ is non-empty. Then there  is a bijection
\[
\mcS_4^{\PD}(B,w,\lambda)\longleftrightarrow \Tors(\Z^w\otimes_{\Z\pi}\Gamma(\pi_2B)).
\]
\end{thm}

\begin{cor}
	If $\Tors(\Z^w\otimes_{\Z\pi}\Gamma(\pi_2X))=0$, the quadratic $2$-type $Q_X$ of $X$ determines the homotopy type of $X$. 
\end{cor}

In the orientable case, the above torsion group vanishes if $\pi$ has cyclic or quaternion 2-Sylow-subgroups \cite{hambleton-kreck,bauer} whereas in the non-orientable case, the following is a simple non-vanishing result:

\begin{prop}
	\label{prop:2.2.2ii}
	Let $(X,w)$ be a Poincar\'e 4-complex with fundamental group $\pi$. If $\pi_2 (X)$ splits off a free module and there exists an element $g\in\pi$ with $g^2=1$ and $w(g)= -1$, then $\Tors(\Z^w\otimes_{\Z\pi}\Gamma(\pi_2X))\neq 0$. 
\end{prop}
\begin{rem}
	This applies to simple $4$-manifolds such as $X=\RP^4\#k\CP^2$ and $X=(\RP^2\times S^2)\#k\CP^2$ for $k\geq 1$. In these cases the quadratic $2$-type still determines the homotopy type by \cite[Theorem~3]{HKT94} when we restrict to manifolds instead of Poincar\'e complexes. 
	
	Considering $X=\RP^4\#\CP^2$ we have $\pi_2(X)\cong \Z[\Z/2]$ and by the proof of \cref{prop:2.2.2ii} we have $\Tors(\Z^w\otimes_{\Z[\Z/2]} \Gamma(\Z[\Z/2]))\cong \Z/2$. We will show in \cref{prop:rp4cp2} that there is a Poincar\'e complex with the same quadratic $2$-type as $X$ which is not homotopy equivalent to $X$. By the above this implies that it cannot be homotopy equivalent to a manifold. 
\end{rem}
\begin{rem}
For $X=\RP^2\times S^2$ we have that $\pi_2(X)\cong \Z\oplus \Z^w$ and thus $\Gamma(\pi_2(X))\cong \Z\oplus\Z\oplus\Z^w$ and $\Z^w\otimes \Gamma(\pi_2(X))\cong \Z/2\oplus\Z/2\oplus\Z$. Hence up to homotopy equivalence there are at most four Poincar\'e complexes with the same quadratic $2$-type as $X$. For manifolds this case was studied by Kim, Kojima and Raymond \cite{KKR92} and Hambleton and Kreck together with the second author \cite{HKT94}. Among the four homotopy types of Poincar\'e complexes there are three that are realized by manifolds. These are distinguished by their  intersection form on $\Z/2$-homology and either a $\Z/4$-valued quadratic refinement of the equivariant intersection form \cite{KKR92} or a $\Z/8$-valued Arf invariant \cite{HKT94}. Hambleton and Milgram \cite[§3]{HM78} showed that there is a Poincar\'e complex with this quadratic $2$-type which is not homotopy equivalent to a manifold. Its Spivak normal fibration admits no reduction to a vector bundle \cite[§4]{HM78}. Hence up to homotopy equivalence there are precisely four Poincar\'e with the quadratic 2-type of $\RP^2\times S^2$.
\end{rem}

We will recall the necessary background on the equivariant intersection form and its connection to Whitehead's quadratic functor in \cref{sec:background}. In \cref{sec:teichner}, we will then prove the statements from the introduction. 

\vspace{3mm}
\noindent \textbf{Acknowledgements:} The authors thank Matthias Kreck for insightful conversations and Mark Powell for helpful comments on a previous version - this all happened at the Max Planck Institute for Mathematics in Bonn. The first author was funded by the Deutsche Forschungsgemeinschaft (DFG, German Research Foundation) under Germany's Excellence Strategy - GZ 2047/1, Projekt-ID 390685813.

\section{Background} \label{sec:background}
Let $X$ be a finite connected Poincar\'e 4-complex with finite fundamental group $\pi$, orientation character $w:\pi_1X\to \{\pm 1\}$ and chosen fundamental class $[X]\in H_4(X;\Z^w)\cong \Z$. Here the coefficient group $\Z^w$ is $\Z$ viewed as a $\Z\pi$-module with action $gz=w(g)z$

We have an involution on $\Z\pi$ given by $g\mapsto \bar g:=w(g)g^{-1}$ and we use this involution to consider a right $\Z\pi$-module as a left $\Z\pi$-module if necessary. We fix a point $x_0\in X$ and a lift $\wt x_0\in \wt X$. We abbreviate $\pi_2(X,x_0)$ by $\pi_2(X)$ and write $H_*(-)$ for $H_*(-;\Z)$. Using the isomorphism $H^2_{cs}(\wt X)\xrightarrow{-\cap[X]}H_2(\wt X)\xleftarrow{\cong}\pi_2(\wt X)\xrightarrow{\cong}\pi_2(X)$, we can view the equivariant intersection form as a pairing on $\pi_2X$, or, as we prefer in this paper, as a pairing $H^2_{cs}(\wt X) \times H^2_{cs}(\wt X)\to\Z\pi$.
It is given by 
\[
\lambda_X(\alpha,\beta)=\lambda_X(\beta)(\alpha)=\sum_{g\in \pi}(g\alpha\cup \beta)\cap[X])g^{-1}=\sum_{g\in \pi}g\alpha\cap ((\beta\cap[X]))g^{-1}.
\]
The form $\lambda_X$ is hermitian with respect to our involution: For $h_i\in\pi$ we have
\begin{align*}
\lambda_X(h_1\alpha,h_2\beta)&=\sum_{g\in \pi}((gh_1\alpha\cup h_2\beta)\cap[X])g^{-1}\\
&=\sum_{g\in \pi}((h_2^{-1}gh_1\alpha\cup \beta)\cap w(h_2)[X])g^{-1}\\
&=\sum_{g\in \pi}((g\alpha\cup \beta)\cap w(h_2)[X])h_1g^{-1}h_2^{-1}\\
&=h_1\lambda(\alpha,\beta)w(h_2)h_2^{-1} = h_1\lambda(\alpha,\beta)\bar h_2
\end{align*}
and $\lambda_X(\alpha,\beta)=\overline{\lambda_X(\beta,\alpha)}$.

The $2$-type $B$ of $X$ is the 2nd stage of a Postnikov tower for $X$. This means that there is a $3$-equivalence $f\colon X\to B$ and $B$ fibers over $K(\pi,1)$ with fiber $K(\pi_2(X),2)$. Such fibrations are classified by the unique obstruction $k_X$ for finding a section.
\[
k_X\in H^3(K(\pi,1);\pi_2(K(\pi_2(X),2)))=H^3(\pi;\pi_2(X))
\]
is by definition the {\em $k$-invariant} of $X$.
The \emph{quadratic $2$-type} of $X$ is
\[Q_X:=(\pi,w,\pi_2(X),k_X, \lambda_X).\]

\begin{defi}
	Let $A$ be an abelian group. Then $\Gamma(A)$ is the abelian group with generators $v(a)$  for all $a\in A$ and the following relations:
	 \[\{v(-a)-v(a)\mid a\in A\}  \quad \text{ and } 
	\]
	 \[\{v(a+b+c)-v(a+b)-v(b+c)-v(c+a)+v(a)+v(b)+v(c)\mid a,b,c\in A\}.\]
\end{defi}

\begin{lemma}[{\cite[page 62]{whitehead}}]
	\label{lem:gammafree}
	If $A$ is free abelian with basis $B$, then $\Gamma(A)$ is free abelian with basis $\{v(b), v(b+b')-v(b)-v(b')\mid b,b'\in B\}$. Another important case is $\Gamma(\Z/2)\cong\Z/4$. 
\end{lemma}

Let $Y$ be a simply-connected CW-complex and let $\eta\colon S^3\to S^2$ be the Hopf map. It induces a map $\widehat{\eta}\colon \Gamma(\pi_2(Y))\to \pi_3(Y)$ given by $v(\alpha)\mapsto \alpha\circ\eta$.

\begin{thm}
	[{\cite[Sections 10 and 13]{whitehead}}]
	For every simply-connected CW-complex $Y$ we have {\em Whitehead's exact sequence}
	\[\ldots\pi_4(Y)\to H_4(Y)\xrightarrow{b_4}\Gamma(\pi_2(Y))\xrightarrow{\widehat{\eta}} \pi_3(Y)\to H_3(Y;Z)\to 0.\]
\end{thm}

We will now explain how the equivariant intersection form $\lambda_X$ for a finite Poincar\'e 4-complex with finite fundamental group can be viewed as an element of $\Gamma(\pi_2(X))$. This is well-known, but for convenience and to fix notation we sketch the argument here.

Let $A$ free abelian. An element $\phi\in \Hom_\Z(\Hom_\Z(A,\Z),\Z),A)$ is symmetric if for all $f,g\in \Hom_\Z(A,\Z)$ we have $f(\phi(g))=g(\phi(f))$. And we denote the subgroup of symmetric homomorphisms by $\Sym_\Z(\Hom_\Z(A,\Z),A)$. Let $Y$ be a finite simply-connected CW-complex with $H_2(Y)$ free abelian. We define $\Sym_\Z(H^2(Y),H_2(Y))$ as $\Sym_\Z(\Hom_\Z(H_2(Y),\Z),H_2(Y))$ using the canonical isomorphism $H^2(Y)\to \Hom_\Z(H_2(Y),\Z)$.

\begin{thm}[{\cite[p.~96]{whitehead}}]
	\label{thm:cap}
	Let $Y$ be a finite simply-connected CW-complex with $H_2(Y)$ free abelian. There is an isomorphism
	\[\zeta\colon \Sym_\Z(H^2(Y),H_2(Y))\to \Gamma(\pi_2(Y))\]
	such that for every class $z\in H_4(Y)$ we have $\zeta(-\cap z)=b_4(z)$.
\end{thm}

Let $X$ be a finite Poincar\'e 4-complex with finite fundamental group $\pi$ and orientation character $w$. Note that by Poincar\'{e} duality
$H_2(\wt X)\cong \Hom_\Z(H_2(\wt X),\Z)$
and thus $H_2(\wt X)$ is free abelian. Let $\Herm(H^2(\wt X))$ denote the hermitian forms on $H^2(\wt X)$. We have a map
\[\Herm(H^2(\wt X))\to \Sym_\Z(H^2(\wt X),\Hom_\Z(H^2(\wt X), \Z)\]
given by
\[\lambda\mapsto (\beta\mapsto(\alpha\mapsto \ev_0\lambda(\alpha,\beta)) ),\]
where $\ev_0\colon \Z\pi\to\Z$ takes the coefficient at the identity element. This map is injective with image the $\pi$-linear homomorphisms, where we view $\Hom_\Z(H^2(\wt X), \Z)$ as a left $\Z\pi$-module with the involution on $\Z\pi$ coming from $w$.
After identifying $\Hom_\Z(H^2(\wt X), \Z)$ with $H_2(\wt X)$, the intersection form $\lambda_X\in \Herm(H^2(\wt X))$ maps to $-\cap [\wt X]$. Therefore, we consider the intersection form as an element of $\Gamma(\pi_2(X))$ as follows.

\begin{defi}
	\label{cor:intform2}
	Let $X$ be a finite Poincar\'e 4-complex with finite fundamental group $\pi$ and orientation character $w$. Then we can consider the element \[\intf_X:=\zeta(-\cap [\wt X])=b_4([\wt X])\in \Gamma(\pi_2(X)).\]
\end{defi}

\begin{cor}
	\label{cor:intform}
	Let $X$ be a finite Poincar\'e 4-complex with finite fundamental group and $2$-type $B$. Let $f\colon X\to B$ be a $3$-equivalence. Then there is an isomorphism $\Gamma(\pi_2(X))\cong H_4(\wt B,\Z)$ mapping $\intf_X$ to $\wt f_*[\wt X]$.
\end{cor}
\begin{proof}
	Since $B$ is the $2$-type of $X$, we have $\pi_3(B)=\pi_4(B)=0$. Consider the diagram:
	\[\xymatrix{
		\pi_4(X)\ar[r]\ar[d]&H_4(\wt X)\ar[r]\ar[d]^{\wt f_*}&\Gamma(\pi_2(X,x_0))\ar[r]\ar[d]^\cong&\pi_3(X)\ar[r]\ar[d]&0\\
		0\ar[r]&H_4(\wt B)\ar[r]^-\cong&\Gamma(\pi_2(B,f(x_0)))\ar[r]&0&	}\]
	From the commutativity of the diagram and \cref{cor:intform2} it follows that $\intf_X$ is mapped to $\wt f_*[\wt X]$ under the composition of the two isomorphisms.
\end{proof}

\begin{cor}
	\label{cor:x}
Let $X=K\cup_\alpha D^4$ be a finite Poincar\'e complex with finite fundamental group, with orientation character $w$ and with $3$-skeleton $K$. Under the map $\Gamma(\pi_2X)=\Gamma(\pi_2K)\to \pi_3(K)$ the element $\intf_X$ is mapped to $N^w\alpha$ (up to a unit), where $N^w=\sum_{g\in\pi}w(g)g\in\Z\pi$.
\end{cor}
\begin{proof}
We can assume that $\pi_2X$ is nontrivial since otherwise $X$ is homotopy equivalent to either a sphere $S^4$ or a projective plane $\RP^4$.
The Whitehead sequences for $X$ and $K$ fit into the following commutative diagram:

\[\xymatrix{
	H_4(\wt X)\ar[r]^-{b_4}&\Gamma(\pi_2X)\ar[r]^-{\widehat{\eta}}&\pi_3(X)\ar[r]&H_3(\wt X)\ar[r]&0\\	
	0=H_4(\wt K)\ar[r]\ar[u]&\Gamma(\pi_2K)\ar@{=}[u]\ar[r]^-{\widehat{\eta}}&\pi_3(K)\ar[r]\ar[u]&H_3(\wt K)\ar[u]\ar[r]&0\\
	&&\pi_4(X,K)\ar[u]\ar[r]^-\cong&H_4(\wt X,\wt K)\cong\Z\pi\ar[u]&\\
	&&\pi_4(X)\ar[u]\ar[r]&H_4(\wt X)\cong\Z\ar[u]^{N^w}&
	}\]
There is a choice of isomorphism $\Z\pi\cong \pi_4(X,K)$ such that $1$ is mapped to $\alpha$ under $\pi_4(X,K)\to \pi_3(K)$. Since $\lambda_X$ comes from $H_4(\wt X)$, the image $\widehat{\eta}(\lambda_X)\in \pi_3(K)$ lies in the kernel of the maps to $\pi_3(X)$ and $H_3(\wt K)$. Therefore, by a quick diagram chase, $\widehat{\eta}(\intf_X)$ is a multiple of $N^w\alpha$ and since $(-\cap [\wt X])\in \Sym_\Z(H^2(\wt X),H_2(\wt X,\Z))$ is an isomorphism and hence primitive, it is mapped to $N^w\alpha$ up to a unit.	
\end{proof}

In \cref{sec:teichner} we will need the following results about the transfer map.

Let $U\leq \pi$ have finite index. Let $X$ be a CW-complex with fundamental group $\pi$ and let $\widehat{X}$ denote the finite covering of $X$ with fundamental group $U$. Then $X$ and $\widehat{X}$ have the same universal covering and let $C_*$ denote its cellular $\Z\pi$-chain complex. We can consider it as a right $\Z\pi$-module using the involution on $\Z\pi$. If $M$ is a $\Z\pi$-module we can restrict the action to $U$ and consider the projection
\[C_*\otimes_{\Z U}M\xrightarrow{p}C_*\otimes_{\Z\pi}M\]
inducing
\[p_*\colon H_*(\widehat{X};M)\to H_*(X;M).\]
On the chain level we obtain a map in the other direction by
\[\tr\colon C_*\otimes_{\Z\pi}M\to C_*\otimes_{\Z U}M,\quad c\otimes m\mapsto \sum_{Ug\in U\backslash \pi}cg^{-1}\otimes gm.\]
This map induces a map
\[\tr_*\colon H_*(X;M)\to H_*(\widehat{X};M)\]
and similarly one constructs a map
\[\tr^*\colon H^*(\widehat X;M)\to H^*(X;M).\]
\begin{lemma}
\label{lem:transfer}
These transfer maps have the following properties:
\begin{enumerate}
	\item $p_*\circ \tr_*$ and $\tr^*\circ p^*$ are multiplication by the index $[U:\pi]$.
	\item If $U\leq \pi$ is normal, then $\pi/U$ acts on $\widehat{X}$ and also on $H_*(\widehat X;M)$ and $H^*(\widehat X;M)$. In this case, $tr_*\circ p_*$ and $p^*\circ tr^*$ are multiplication by $\sum_{g\in \pi/U}g$.
	\item If $\pi$ acts trivially on a commutative ring $R$, then $tr^*$ is an $H^*(X;R)$-module homomorphism, i.e.
	\[tr^*(p^*(x)\cup y)=x\cup tr^*(y)\]
	for all $x\in H^*(X;R)$ and $y\in H^*(\widehat{X};R)$.
	\item The transfer commutes with Kronecker-products and Steenrod-squares:
	\[\langle tr^*(y),b\rangle=\langle y,tr_*b\rangle\in R\]
	and
	\[Sq^i(tr^*(y))=tr^*(Sq^i(y))\in H^*(X;\Z/2).\]
\end{enumerate}
\end{lemma}
\begin{proof}
(1) and (2) follow directly from the definitions. For (3) and (4) one uses that for trivial modules $tr_*$ and $tr^*$ are induced by a stable map $tr\colon \Sigma^\infty X\to \Sigma^\infty\widehat{X}$ of suspension spectra, see \cite[Chapter 4]{adams}.
\end{proof}

\section{The \texorpdfstring{$\CP^2$-}{CP2-}stable classification}
\label{sec:teichner}
In this section we will recall the proof of \cite[Theorem 1.1]{hambleton-kreck} and then prove \cref{thm:main,thm:hk,thm:teichner}. At the end of the section we will prove \cref{prop:2.2.2ii} and \cref{prop:rp4cp2} which was mentioned in the introduction.

Let $X, Y$ be finite Poincar\'e 4-complexes with finite fundamental group $\pi$. Assume that $X$ and $Y$ have isomorphic quadratic $2$-types:
\[
Q_X=(\pi, w,\pi_2(X),k,\lambda_X)\cong Q_Y.
\]
In the following we will always assume that $\pi_2(X)\neq 0$ since otherwise $X$ is homotopy equivalent to either a sphere $S^4$ or a projective plane $\RP^4$.
Let $B$ be the $2$-type of $X$ and let $f\colon X\to B$ be a $3$-equivalence.
From the results of \cref{sec:background}, in particular \cref{cor:intform}, it follows that there is a $3$-equivalence $g\colon Y\to B$ such that $\wt g_*[\wt Y]=\wt f_*[\wt X]$. 

If in this situation there is a map $h\colon X\to Y$ making the diagram
\[\xymatrix{
	X\ar[rr]^h\ar[rd]_f&&Y\ar[ld]^g\\&B&}\]
homotopy commutative, then $h$ is an isomorphism on $\pi_1, \pi_2$. By \cref{cor:intform} there is an isomorphism $\Gamma(\pi_2(Y))\cong H_4(\wt B)$ mapping $\lambda_Y$ to $\wt g_*[\wt Y]$ and by assumption, $\wt g_*[\wt Y]=\wt f_* [\wt X]=\wt g_*\wt h_* [\wt X]$. Since $H_4(\wt Y)\cong \Z$, $H_4(\wt B)$ is torsion-free and $\wt g_*[\wt Y]$ is non-trivial, this implies $\wt h_*[\wt X]=[\wt Y]$. Thus by Poincar\'e duality $\wt h$ is an isomorphism on homology of universal coverings in all degrees and hence $h$ is a homotopy equivalence.  

\begin{lemma}
	\label{lem:2.1.1}
	A map $h$ as above exists if and only if $f_*[X]=g_*[Y]\in H_4(B;\Z^w)$.
\end{lemma}
\begin{proof}
	This result is proved by Hambleton and Kreck \cite[Lemma 1.3]{hambleton-kreck} in the oriented case. Since the proof in the non-oriented case works the same we will not give any details here. However, we want to point out the following two things.
	
	Firstly, there is a minor mistake at the beginning of the proof of \cite[Lemma 1.3]{hambleton-kreck}, where the authors construct $B$ by attaching cells of dimension $4$ and higher to $Y$, i.e.\ $g\colon Y\to B$ is the inclusion. Then they say that they want to extend
	\[X^{(3)}\xrightarrow{f^{(3)}}B^{(3)}=Y^{(3)}\subseteq Y\]
	over the $4$-cell of $X$.
	
	In fact, using the vanishing of obstructions in cohomology they can only extend the map $X^{(2)}\to Y$, because they might have to change the given map on the $3$-cells in order to extend over the $4$-cell. But also such an extension $h$ suffices for the proof because the given (trivial) homotopy between $f$ and $g\circ h$ on $X^{(2)}\times I$ extends to $X\times I$ using $\pi_3(B)=\pi_4(B)=0$.
	
	Secondly, to get the non-oriented case, note that Hambleton and Kreck show for $w\equiv 1$ that the obstruction in
	\[H^4(X;\pi_3Y)\cong H_0(X;\pi_3Y\underset{\Z\pi}{\otimes}\Z^w)\cong \pi_3Y\underset{\Z\pi}{\otimes}\Z^w\cong \pi_4(B,Y)\underset{\Z\pi}{\otimes}\Z^w\cong H_4(B,Y;\Z^w)\]
	is given by the image of $f_*[X]$ under $H_4(B;\Z^w)\to H_4(B,Y;\Z^w)$ and therefore vanishes if $f_*[X]=g_*[Y]$. All these isomorphisms continue to hold for non-trivial $w$.
\end{proof}

Let $X$ be a finite Poincar\'e complex with orientation character $w$, finite fundamental group $\pi$, intersection form $\lambda_X$ and with a fixed $3$-equivalence $f\colon X\to B$, where $B$ is $3$-coconnected.  

Let $\lambda$ be a hermitian form on $\pi_2(B)$.
Recall that $\mcS_4^{\PD}(B,w,\lambda)$ denotes the set of $4$-dimensional $B$-polarized Poincar\'e complexes $(Y,g)$ with orientation character w, such that $\lambda$ maps to $\lambda_{Y}$, up to homotopy equivalence over $B$. Assume $(X,f)\in \mcS_4^{\PD}(B,w,\lambda)$. 
By \cref{lem:2.1.1}, sending $(Y,g)$ to $g_*[Y]-f_*[X]$ defines an injection
\[\mcS_4^{\PD}(B,w,\lambda)\to H_4^w(B):=H_4(B;\Z^w).\]

We want to compute the image of this map. Under the isomorphism $H_4(\wt B)\cong \Gamma{\pi_2(B)}$, the element $\wt f_*[\wt X]$ is determined by $\lambda=f^*\lambda_{X}$ by \cref{cor:intform}. Furthermore, recall that $\tr_*[X]=[\wt X]$. Therefore $|\pi|(g_*[Y]-f_*[X])=p_*\tr_*(g_*[Y]-f_*[X])=0$, i.e.\ the above image lies in the torsion subgroup of $H_4^w(B)$.

Write $X$ as $K\cup_\alpha D^4$, where $K$ is a $3$-complex. Since $K$ is a $3$-dimensional complex, $H_3^w(K)$ is a subgroup of a free abelian group and thus torsion-free. Furthermore, $H^1(X)\cong H^1(\pi)\cong \Hom_\Z(\pi,\Z)=0$. Using this, from the exact sequence
\[\xymatrix{0\ar[r]& H_4^w(X)\ar@{=}[d]\ar[r] &H_4^w(X,K)\ar@{=}[d]\ar[r]&H_3^w(K)\ar[r]&H_3^w(X)\ar[d]^\cong_{\PD^{-1}}\\&\Z&\Z&&H^1(X)}\]
we obtain $H_3^w(K)=0$. Therefore, $H_4^w(B)\to H_4^w(B,K)$ is an isomorphism. Since $B$ is $3$-coconnected, $\pi_4(B,K)\to \pi_3(K)$ is an isomorphism. Combining these, we obtain the very useful isomorphism
\begin{equation}
\label{eq:useful}
H_4^w(B)\xrightarrow{\cong}H_4^w(B,K)\xleftarrow{\cong}\pi_4(B,K)\otimes_{\Z\pi}\Z^w\xrightarrow{\cong}\pi_3(K)\otimes_{\Z\pi}\Z^w.\end{equation}
Consider the commutative diagram
\[\xymatrix
{
	H^w_4(X)\ar[r]\ar[d]^{f_*}&H^w_4(X,K)\ar[d]^{f_*}&\pi_4(X,K)\otimes_{\Z\pi}\Z^w\ar[l]\ar[d]^{f_*}\ar[dr]^{\partial}&\\
	H^w_4(B)\ar[r]^\cong&H^w(B,K)&\pi_4(B,K)\otimes_{\Z\pi}\Z^w\ar[l]_\cong\ar[r]^\cong_\partial&\pi_3(K)\otimes_{\Z\pi}\Z^w
}
\]
and view the top cell of $X$ as an element in $\pi_4(X,K)$. Then it has boundary $\alpha$ and its image in $H^w_4(X,K)$ is also the image of $[X]\in H_4^w(X)$. Hence the composition of the horizontal isomorphisms \eqref{eq:useful} maps $f_*[X]$ to $[\alpha]\otimes 1$.

Finally, we use Whitehead sequences for $\wt X$ and $\wt K$:
\[\xymatrix{
	\pi_4(\wt X)\ar[r]^-0&H_4(\wt X)\ar[r]^-{b_4}&\Gamma(\pi_2(X))\ar[r]^-{\widehat{\eta}}&\pi_3 (X)\ar[r]&H_3(\wt X)\ar@{=}[r]&0\\
	&0\ar[u]\ar[r]&\Gamma(\pi_2(K))\ar@{=}[u]\ar[r]^-{\widehat{\eta}}&\pi_3 (K)\ar[u]\ar[r]&H_3(\wt K)\ar[u]\ar[r]&0\\
	&&&\pi_4(\wt X,\wt K)\ar[u]\ar[r]^\cong&H_4(\wt X,\wt K)\ar[r]^-\cong\ar[u]&\Z\pi\\
	&&&\pi_4(\wt X)\ar[u]^0\ar[r]^0&H_4(\wt X)\ar[u]&}\]
The image of the composition $H_4(\wt X)\xrightarrow{b_4}\Gamma(\pi_2 (X))=\Gamma(\pi_2(K))\xrightarrow{\widehat{\eta}}\pi_3(K)$ lies in the kernel of the map to $\pi_3(X)$ and thus inside $\pi_4(\wt X,\wt K)\cong H_4(\wt X,\wt K)\cong \Z\pi$. We see that it fits into a short exact sequence
\[0\to H_4(\wt X)\to H_4(\wt X,\wt K)\to H_3(\wt K)\to 0\]
and we obtain the following commutative diagram of short exact sequences:
\[\xymatrix{
	0\ar[r]&H_4(\wt X)\ar[r]\ar[d]^{b_4}&\Z\pi\ar[r]\ar[d]&H_3(\wt K)\ar[r]\ar@{=}[d]&0\\
	0\ar[r]&\Gamma(\pi_2(K))\ar[r]&\pi_3(K)\ar[r]&H_3(\wt K)\ar[r]&0}\]
The fundamental class $[\wt X]$ maps downward to $\lambda_X\in\Gamma(\pi_2X)=\Gamma(\pi_2K)$ and, by \cref{cor:x}, to the right to $\pm N^w:=\sum_{g\in\pi}w(g)g$. Note that $H_4(\wt X)\cong \Z^w$. By the upper short exact sequence above, $H_3(\wt K)\cong \Z\pi/N^w$.
Tensoring the diagram with $\Z^w$ over $\Z\pi$ we obtain the following diagram of the long exact Tor-sequences:{\small
\[\xymatrix@C=.7pc@R=.8pc{
	&&\Z\ar@{=}[d]\ar[r]^{\cdot|\pi|}&\Z\ar@{=}[d]&&\\
	0\ar[r]&\Tor_1^{\Z\pi}(\Z\pi/N^w,\Z^w)\ar[r]\ar@{=}[d]&\Z^w\otimes_{\Z\pi}\Z^w\ar[d]\ar[r]^-{N^w\otimes \id}&\Z\pi\otimes_{\Z\pi}\Z^w\ar[r]\ar[d]&\Z\pi/N^w\otimes_{\Z\pi}\Z^w\ar[r]\ar@{=}[d]&0\\
	0\ar[r]&\Tor_1^{\Z\pi}(\Z\pi/N^w,\Z^w)\ar[r]&\Gamma(\pi_2(K))\otimes_{\Z\pi}\Z^w\ar[r]&\pi_3(K)\otimes_{\Z\pi}\Z^w\ar[r]&\Z\pi/N^w\otimes_{\Z\pi}\Z^w\ar[r]&0}\]}
The top row shows that $\Tor_1^{\Z\pi}(\Z\pi/N^w,\Z^w)=0$ and $\Z\pi/N^w\otimes_{\Z\pi}\Z^w=\Z/|\pi|\Z$. Using that $\pi_3K\otimes_{\Z\pi}\Z^w\cong H_4^w(B)$ and $\pi_2(B)\cong \pi_2(K)$, the bottom row can be identified with the short exact sequence:
\begin{equation}
	\label{eq:1}
	0\to \Gamma(\pi_2(B))\otimes_{\Z\pi}\Z^w\xrightarrow{i} H_4^w(B)\xrightarrow{q}\Z/|\pi|\to 0
\end{equation}
with $i(\lambda\otimes 1)=|\pi|f_*[X]$ and $q(f_*[X])=1+|\pi|\Z$, exactly as in the oriented case \cite[page 89]{hambleton-kreck}. From here \cref{thm:hk} follows as in the proof of \cite[Theorem~1.1]{hambleton-kreck}.

Now we are nearly ready to prove the following theorem which completes the proof of \cref{thm:teichner}. Recall that $(Y,g)\mapsto g_*[Y]-f_*[X]$ defines an injection $\mcS_4^{\PD}(B,w,f_*\lambda_X)\to H_4^w(B)$ with image inside the torsion subgroup.
\begin{thm}
	\label{thm:mainteichner}
	For all $(Y,g)\in \mcS_4^{\PD}(B,w,\lambda)$ we have $q(g_*[Y]-f_*[X])=0$ and $(Y,g)\mapsto i^{-1}(g_*[Y]-f_*[X])$ defines a bijection
	\[\mcS_4^{\PD}(B,w,\lambda)\leftrightarrow \Tors(\Gamma(\pi_2B)\otimes_{\Z\pi}\Z^w).\]
\end{thm}
Before proving \cref{thm:mainteichner} we will first show how it implies \cref{thm:main}.
\begin{proof}[Proof of \cref{thm:main}]
It remains to show that for a given closed, connected $4$-manifold $M$ with fundamental group $\pi$ and orientation character $w$ the element $c_*[M]\in H_4(\pi;\Z^w)$ is determined by the stable quadratic $2$-type of $M$. Let $B$ denote the Postnikov $2$-type of $M$ and consider the sequence $M\xrightarrow{f} B\xrightarrow{c} B\pi$. 
Considering the following diagram and following the construction of \eqref{eq:1}, we see that the map $i$ in \eqref{eq:1} can be identified with the canonical map $H_4(\wt B)\otimes_{\Z\pi}\Z^w\to H_4^w(B)$.
\[\xymatrix{
	&&&0&0&\\
	&0\ar[r]&\Gamma(\pi_2(\wt K))\ar[r]^-{\widehat{\eta}}&\pi_3 (\wt K)\ar[u]\ar[r]&H_3(\wt K)\ar[u]\ar[r]&0\\
	&&&\pi_4(\wt B,\wt K)\ar[u]_\cong\ar[r]^\cong&H_4(\wt B,\wt K)\ar[u]&\\
	&&&0\ar[u]\ar[r]&H_4(\wt B)\ar[u]&\\
&&&&0\ar[u]&
}\]
Now given another closed, connected $4$-manifold $N$ with orientation character $w$ and a $3$-equivalence $g\colon N\to B$ such that $g^*\lambda_N=f^*\lambda_M$, then $f_*[M]-g_*[N]$ is contained in $\Tors(\Gamma(\pi_2(B))\otimes_{\Z\pi}\Z^w)\cong \Tors(H_4(\wt B)\otimes_{\Z\pi}\Z^w)$ by \cref{thm:mainteichner}. Since the composition $\wt B\to B\xrightarrow{c} B\pi$ is trivial, we thus have $c_*f_*[M]=c_*g_*[N]\in H_4^w(B\pi)$ as claimed.
\end{proof}
For the proof of \cref{thm:mainteichner} we need the following lemma.
\begin{lemma}
	\label{lem:mainteichner}
	If $w\equiv 1$ or if the $2$-Sylow subgroup of $\pi$ has order bigger than 2, there exists a homomorphism
	\[\kappa\colon\Gamma(\pi_2B)\otimes_{\Z\pi}\Z^w\to \Z~\text{with}~\lambda\otimes 1\mapsto 1.\]
\end{lemma}
We will first prove \cref{thm:mainteichner} assuming \cref{lem:mainteichner}, and then we will prove \cref{lem:mainteichner}.
\begin{proof}[Proof of \cref{thm:mainteichner}]
	We will first show that $q(g_*[Y]-f_*[X])=0$.
	
	In the case that $w\equiv 1$ or the $2$-Sylow subgroup of $\pi$ has order more than $2$, we will show that $i(\Tors(\Z^w\otimes_{\Z\pi}\Gamma(\pi_2B)))=\Tors(H_4^w(B))$. This proves $q(g_*[Y]-f_*[X])=0$ in this case. The image of $\Tors(\Z^w\otimes_{\Z\pi}\Gamma(\pi_2B))$ is obviously contained in $\Tors(H_4^w(B))$ and we have to show the other inclusion.
	
	Let $x\in \Tors(H_4^w(B))$ be given. Because $H_4(\wt{B})\cong \Gamma(\pi_2B)$ is torsion free, we see that $|\pi|x=p_*\tr_*(x)=0$. From \eqref{eq:1} and $q(f_*[X])=1\in\Z/|\pi|$ we see that there exists $k\in\bbN$ and $y\in \Gamma(\pi_2 B)\otimes_{\Z\pi}\Z^w$ such that $x+kf_*[X]=i(y)$. It follows that
	\[|\pi|i(y)=|\pi|x+k|\pi|f_*[X]=ki(\lambda\otimes 1)\]
	Thus, $|\pi|y=k(\lambda\otimes 1)$ and by \cref{lem:mainteichner} we have $|\pi|\kappa(y)=k\kappa(\lambda\otimes 1)=k$. Therefore, $|\pi|$ divides $k$ and
	\[x=i(y-\frac{k}{|\pi|}(\lambda\otimes 1))\in \im(i).\]
	
	We now consider the case $\pi=Q\rtimes\Z/2$ with $|Q|$ odd and $w=p_2$ the projection onto $\Z/2\cong\{\pm 1\}$. Since $f_*[X]\in H_4^w(B)$ maps to a generator of $\Z/|\pi|\cong \Z/|Q|\oplus\Z/2$ under $q$ there is $x\in \Z/|Q|$ with $q(f_*[X])=(x,1+2\Z)\in \Z/|Q|\oplus \Z/2$. For any $(Y,g)\in \mcS_4^{\PD}(B,w,\lambda)$ there is in the same way $y\in \Z/|Q|$ with $q(g_*[Y])=(y,1+2\Z)$ and hence $q(g_*[Y]-f_*[X])=(y-x,0)$.
	
	Taking double coverings with respect to the subgroup $Q\leq\pi$ we get a commutative square
	\[\xymatrix{
		\widehat{X}\ar[r]^{\widehat{f}}\ar[d]&\widehat{B}\ar[d]^p\\
		X\ar[r]^f&B}\]
	where $p_*\widehat{f}_*[\widehat{X}]=2f_*[X]$. So we may conclude from the following commutative diagram that $\widehat{q}(g_*[\widehat{Y}])=y$.
		\[\xymatrix{
			0\ar[r]&\Gamma(\pi_2\widehat{B})\otimes_{\Z Q}\Z\ar[r]^-{\widehat i}\ar[d]&H_4(\widehat{B})\ar[d]^{p_*}\ar[r]^-{\widehat{q}}&\Z/|Q|\ar[r]\ar[d]^{\cdot 2}&0\\
			0\ar[r]&\Gamma(\pi_2B)\otimes_{\Z\pi}\Z^w\ar[r]^-i&H_4^w(B)\ar[r]^-q&\Z/|\pi|\ar[r]&0}\]
	Since $\widehat{X}$ is an orientable Poincar\'e complex we already know that \cref{thm:mainteichner} holds for $\widehat{X}$. Therefore, $0=\widehat{q}(\widehat{g}_*[\widehat{Y}]-\widehat{f}_*[\widehat{X}])=y-x\in \Z/|Q|$ and thus also $q(g_*[Y]-f_*[X])=0$.
	
	By the above, the map $\mcS_4^{\PD}(B,w,\lambda)\rightarrow \Tors(\Gamma(\pi_2B)\otimes_{\Z\pi}\Z^w)$ is well defined and we already now that it is injective. Showing surjectivity is easy since we can use \cite[page 89]{hambleton-kreck} to reattach the top cell of $X$ as follows. Recall that $X=K\cup_\alpha D^4$. For a given element $\beta\in \Tors(\Gamma(\pi_2B)\otimes_{\Z\pi}\Z^w)$ pick a preimage $\widehat\beta\in \Gamma(\pi_2B)\subseteq \pi_3(K)$ and attach a $4$-cell to $K$ along a representative of $\widehat\beta+\alpha$. The resulting complex $X_\beta$ has the same intersection form as $X$ and also $H_3(\wt{X_\beta})=0$ because $\widehat\beta$ maps trivially to $H_3(\wt K)$. As in \cite{hambleton-kreck}, the map $f|_K$ extends to a $3$-equivalence $f_\beta\colon X_\beta\to B$ such that $f_{\beta*}[X_\beta]=f_*[X]+\beta$.
\end{proof}
For the proof of \cref{lem:mainteichner} we need the following constructions of maps from $\Gamma(\pi_2(X))=\Gamma(\pi_2(B))$ to $\Z^w$.
Consider the isomorphism
\[\zeta^{-1}\colon \Gamma(\pi_2(X))\to \Sym_\Z(H^2(\wt X),H_2(\wt X))\]
from \cref{thm:cap}.
Precomposing with the inverse of $\PD\colon H^2(\wt X)\xrightarrow{-\cap[\wt X]} H_2(\wt X)$, we get a map
\begin{equation}
\label{eq:kappa}\kappa'\colon \Gamma(\pi_2(X))\to\Hom_\Z(H_2(\wt X),H_2(\wt X)).\end{equation}
The construction of $\kappa'$ and the following two lemmas below are due to Stefan Bauer \cite{bauer}.
\begin{lemma}
	If we let $\pi$ act on $\Hom_\Z(H_2(\wt X),H_2(\wt X))$ by $g\phi=w(g)(g_*\circ\phi\circ g_*^{-1})$, then $\kappa'$ is equivariant. In particular,
	\[\trace\circ \kappa'\colon \Gamma(\pi_2(X))\to \Z^w\]
	is equivariant.
\end{lemma}
\begin{proof}
	If we let $\pi$ act on $\Sym(H^2(\wt X),H_2(\wt X))$ by $g\psi=g_*\circ \psi\circ g^*$, then by naturality $\zeta$ is equivariant.
	We have
	\[g_*(g^*\alpha\cap [\wt X])=\alpha\cap g_*[\wt X]=w(g)(\alpha\cap[\wt X])\]
	and thus $g^*\circ \PD^{-1}=w(g)(\PD^{-1}\circ g^{-1}_*)$. This proves the lemma since $\kappa'(x)=\zeta(x)\circ \PD^{-1}$.
\end{proof}
\begin{lemma}
	\label{lem:kappa'}
	When we view $\lambda_X$ as an element of $\Gamma(\pi_2(X))$ as in \cref{cor:intform2}, then
	\[\kappa'(\lambda_X)=\id_{H_2(\wt X)}.\]
\end{lemma}	
\begin{proof}
	Recall that $\lambda_X$ is the image of $-\cap [\wt X]$ under $\zeta$. Thus, precomposing with the inverse of $\PD=-\cap[\wt X]$ gives the identity on $H_2(\wt X)$.
\end{proof}

A central involution in a group $G$ is an element of order 2 in the center of $G$.
\begin{lemma}
	\label{lem:2.3.2}
	Let $G$ be a $2$-group, $w\colon G\to \Z/2$. If $w\equiv 1$ or $|G|>2$, there exists a central involution $\tau\in G$ with $w(\tau)=1\in\{\pm 1\}$.
\end{lemma}
\begin{proof}
	Since $G$ is a $2$-group there exists a central involution. We are done if $w(\tau)=1$. Otherwise the map $w$ splits and since $\tau$ is central, we have $G=\langle \tau\rangle\times \ker(w)$. Since $|G|>2$ there again exists a central involution $\tau'\in \ker (w)$ and this is still central in $G=\langle \tau\rangle\times \ker(w)$.
\end{proof}

Let $\tau_*\colon \Hom_Z(H_2(\wt X),H_2(\wt X))\to \Hom_Z(H_2(\wt X),H_2(\wt X))$ be given by precomposition with the action of $\tau$ on $\wt X$.
\begin{lemma}
	\label{lem:2.3.3}
	The map
	\[\trace\circ \tau_*\circ \kappa'\colon \Gamma(\pi_2(X))\to \Z^w\]
	is equivariant and the element $\lambda_X$ is mapped to $-2$. 
\end{lemma}
\begin{proof}
	Since $\kappa'$ and the trace are equivariant, we have to show that $\tau_*$ is equivariant. This is the case, because the group element $\tau$ is central.
	Since $\kappa'(\lambda_X)=\id_{H_2(\wt X)}$ by \cref{lem:kappa'}, we have $\tau_*(\kappa'(\lambda_X))=\tau\colon H_2(\wt X)\to H_2(\wt X; \Z)$. Since $w(\tau)=1$ and $\tau$ acts fixed point free on $\wt X$, the Lefschetz formula yields
	\[0=\trace(\tau_{H_0(\wt X)})+\trace(\tau_{H_4(\wt X)})+\trace(\tau_{H_2(\wt X)})=2+\trace(\tau_*(\kappa'(\lambda_X)))\]
	and therefore $\trace(\tau_*(\kappa'(\lambda_X)))=-2$.
\end{proof}
We want to show that the image of $\trace\circ \tau_*\circ\kappa'$ is always even. For this we need the following lemma. This is well-known and for example follows from \cite[VII~Theorem~7.4]{bredon-trans}. For the reader's convenience we give a proof.
\begin{lemma}
	\label{lem:even}
	Let $\tau\in\pi$ have order 2. Then the symmetric bilinear form
	\[H^2(\wt{X};\Z/2)\otimes H^2(\wt{X};\Z/2)\to \Z/2,\quad x\otimes y\mapsto \langle \tau x\cup y,[\wt{X}]\rangle\]
	is even, i.e.\ $\langle \tau x, x\rangle=0$ for all $x\in H^2(\wt X;\Z/2)$.
\end{lemma}
\begin{proof}
	Let $q\colon \wt{X}\to \widehat{X}:=\wt{X}/\tau$ be the double covering. Then $q^*\circ \tr^*(x)=x+\tau x$ and
	\begin{align*}
		\langle q^*\circ \tr^*(x)\cup x,[\wt X]\rangle &=\langle \tr^*(q^*\circ \tr^*(x)\cup x),[\widehat{ X}]\rangle\\
		&=\langle \tr^*(x)\cup \tr^*(x),[\widehat{X}]\rangle\\
		&=\langle Sq^2(\tr^*(x)),[\widehat{X}]\rangle\\
		&=\langle Sq^2(x),\tr_*[\widehat{X}]\rangle\\
		&=\langle x\cup x,[\wt X]\rangle.
	\end{align*}
	Here we used the identities from \cref{lem:transfer} (with $p=q$).
	It follows that \[\langle x\cup \tau x,[\wt X]\rangle=\langle x\cup (q^*\circ \tr^*(x)+x),[\wt X]\rangle=2\langle x\cup x,[\wt X]\rangle=0\]
	for all $x\in H^2(\wt X;\Z/2)$.
\end{proof}
\begin{lemma}
	\label{cor:even}
	The image of $\trace\circ\tau_*\circ \kappa'$ is contained in $2\Z^w$.
\end{lemma}
\begin{proof}
	The trace fits into the following commutative diagram:
	\[\xymatrix{
			\Hom_\Z(H_2(\wt X),H_2(\wt X))\ar[rr]^{\trace}&&\Z\\
			&H^2(\wt X)\otimes H_2(\wt X)\ar[lu]^{\phi_1}\ar[ru]^{\phi_2}&}\]
	where $\phi_1$ is the slant map and $\phi_2$ the evaluation, i.e.\ $\phi_1(\alpha\otimes x):=(y\mapsto (\alpha\cap y)x)$ and $\phi_2(\alpha\otimes x):=\alpha \cap x$.
	
	Let $\{\beta_i\}$ be a $\Z$-basis of $H_2(\wt X)$. Every element in $H^2(\wt X)\otimes H_2(\wt X)$ can be written as $\sum_{i,j}\lambda_{i,j}\beta_i\otimes(\beta_j\cap[\wt X])$ for some $\lambda_{i,j}\in \Z$. The map
	\[\Psi:=\phi_1(\sum_{i,j}\lambda_{i,j}\beta_i\otimes(\beta_j\cap[\wt X]))\circ \PD\colon H^2(\wt X)\to H_2(\wt X)\]
	is given by $\Psi(\alpha)=\sum_{i,j}\lambda_{i,j}((\alpha\cup \beta_i)\cap[\wt X])(\beta_j\cap[\wt X])$. Recall that it is symmetric if $\alpha'\cap \Psi(\alpha)=\alpha\cap \Psi(\alpha')$. Let $\{\beta_i^*\}$ be the basis dual to $\{\beta_i\}$ in the sense that $(\beta_i^*\cup\beta_j)\cap[\wt X]=\delta_{ij}$. We see that if $\Psi$ is symmetric, then
	\begin{align*}
	\lambda_{i',j'}&=\sum_{i,j}\lambda_{i,j}((\beta_{i'}^*\cup \beta_i)\cap[\wt X])((\beta^*_{j'}\cup\beta_j)\cap[\wt X])\\
	&=\beta_{j'}^*\cap\Psi(\beta^*_{i'})\\
	&=\beta_{i'}^*\cap\Psi(\beta^*_{j'})\\
	&=\sum_{i,j}\lambda_{i,j}((\beta_{j'}^*\cup \beta_i)\cap[\wt X])((\beta^*_{i'}\cup\beta_j)\cap[\wt X])\\
	&=\lambda_{j',i'}
	\end{align*}
	One easily computes
	\[\tau_*(\phi_1(\sum_{i,j}\lambda_{i,j}\beta_i\otimes(\beta_j\cap[\wt X])))=\phi_1(\sum_{i,j}\lambda_{i,j}\tau\beta_i\otimes(\beta_j\cap[\wt X]))\]
	and using that $\tau$ has order two and $w(\tau)=1$, we get
	\[(\tau \beta_i\cup\beta_j)\cap[\wt X]=(\tau^2\beta_i\cup\tau\beta_j)\cap w(\tau)[\wt X]=(\beta_i\cup\tau\beta_j)\cap[\wt X].\]
	Let $\Psi$ be symmetric. By \cref{lem:even} we then get
	\begin{align*}
	\phi_2(\sum_{i,j}\lambda_{i,j}\tau\beta_i\otimes(\beta_j\cap[\wt X]))&=\sum_{i,j}\lambda_{i,j}(\tau\beta_i\cup\beta_{j})\cap[\wt X]\\
	&=\sum_i(\tau\beta_i\cup\beta_{i})\cap[\wt X]+2\sum_{i<j}\lambda_{i,j}(\tau\beta_i\cup\beta_j)\cap[\wt X]\\
	&=2\sum_i(\beta_i\cup\beta_{i})\cap[\wt X]+2\sum_{i<j}\lambda_{i,j}(\tau\beta_i\cup\beta_j)\cap[\wt X]
	\end{align*}
	This implies the lemma since by definition of $\kappa'$ every element in its image is coming from $\Sym_\Z(H^2(\wt X),H_2(\wt X))$.
\end{proof}

\begin{proof}[Proof of \cref{lem:mainteichner}]
	Since $\Z^w\otimes_{\Z\pi} \Z^w=\Z$, $f_*\colon \pi_2(X)\to \pi_2(B)$ is an isomorphism and $f^*\lambda=\lambda_X$, it suffices to show that there exists a $\Z\pi$-homomorphism
	\[\kappa\colon \Gamma(\pi_2(X))\to\Z^w\]
	with $\kappa(\lambda_X)=1$.
	
	We have $\operatorname{rank}(H_2(\wt X))=\chi(\wt{X})-2=|\pi|\chi(X)-2$, where $\chi$ denotes the Euler characteristic. Thus it suffices to show that there exist homomorphisms $\kappa_i\colon \Gamma(\pi_2(X))\to \Z^w$ such that:
	\begin{enumerate}
		\item $\kappa_1(\lambda_X)=|\pi|$,
		\item $\kappa_2(\lambda_X)=\operatorname{rank}(H_2(\wt X))$ and
		\item $\kappa_3(\lambda_X)$ is odd,
	\end{enumerate}
	because if $\kappa_3(\lambda_X)=2n+1$, then for $\kappa=\kappa_3+n\kappa_2-n\chi(X)\kappa_1$ we have $\kappa(\lambda_X)=1$.
	\textbf{(1):}
	For a $\Z\pi$-module $M$ let $M^{\pi_w}$ denote the submodule consisting of those objects $m$ with $gm=w(g)m$ for all $g\in\pi$. We have $\lambda_X\in \Gamma(\pi_2(X))^{\pi_w}$. Any $\Z$-homomorphism $\Gamma(\pi_2(X))^{\pi_w}\to \Z^w$ is already a $\Z\pi$-homomorphism, since $\pi$ acts on both sides via $w$. Since $X$ is a finite 4-dimensional Poincar\'e complex with finite fundamental group, $\pi_2(X)$ is free abelian. By \cref{lem:gammafree}, $\Gamma(\pi_2(X))$ is therefore free abelian. Capping with $[\wt X]$ gives an isomorphism $H^2(\wt X)\to H_2(\wt X)$ and hence $\lambda_X\in \Gamma(\pi_2(X))^{\pi_w}$ is primitive. Therefore, there exists a $\Z$- and thus $\Z\pi$-linear map
	\[d\colon \Gamma(\pi_2(X))^{\pi_w}\to \Z^w\]
	with $d(\lambda_X)=1$. We define $\kappa_1$ as the composition
	\[\kappa_1\colon \Gamma(\pi_2(X))\xrightarrow{\cdot N^w}(\Gamma(\pi_2(X))^w)^\pi\xrightarrow{d}\Z^w.\]
	\textbf{(2):}
	Define $\kappa_2:=\trace\circ \kappa'$. By \cref{lem:kappa'}, \[\kappa_2(\lambda_X)=\trace(\Id_{H_2(\wt X)})=\operatorname{rank}(H_2(\wt X)).\]
	\textbf{(3):}
	Let $G\leq \pi$ be a $2$-Sylow subgroup and $p\colon X'\to X$ the corresponding covering. Then the map
	\[tr_G^\pi\colon \Gamma(\pi_2X)\otimes_{\Z\pi}\Z^w\to \Gamma(\pi_2X)\otimes_{\Z G}\Z^{p^*w}\]
	has the property that 
	\[tr_G^\pi(\lambda_X\otimes 1)=\sum_{k\in G\backslash \pi}(k(\lambda_X)\otimes k\cdot 1)=[\pi:G](\lambda_X\otimes 1),\]
	where the last equation follows from $k(\lambda_X) = w(k)\cdot \lambda_X \in \Gamma(\pi_2X)$ for all $k\in\pi$.
	Since $[\pi:G]$ is odd and $\lambda_X=\lambda_{X'}$, we can therefore assume that $\pi$ is a $2$-group by composing $\kappa_3$ for $G$ with $tr_G^\pi$.
	Note that $p^*\colon H^1(X;\Z/2)\to H^1(X';\Z/2)$ is injective and thus $p^*w$ is non trivial if and only if $w$ is non trivial.
	
	Let $\kappa'$ be the homomorphism from \eqref{eq:kappa}. By \cref{cor:even}, the image of $\trace\circ\tau_*\circ\kappa'$ is always even and we can define $\kappa_3:=-\frac{1}{2}(\trace\circ\tau_*\circ\kappa')$. By \cref{lem:2.3.3}, $\kappa_3(\lambda_X)=1$.
\end{proof}

\begin{proof}[Proof of \cref{prop:2.2.2ii}]
	By \cite[Theorem 2.1]{hambleton-kreck} we have
	\[\Gamma(\Z\pi)\cong F\oplus \bigoplus_{g\in\pi\setminus\{1\}, g^2=1}\Z\pi/\Z\pi(1-g),\]
	for some free module $F$. It is an easy calculation that this implies
	\[\Tors(\Z^w\otimes_{\Z\pi}\Gamma(\Z\pi))\cong (\Z/2)^r,\]
	where $r$ is the number of $g\in \pi\setminus\{1\}$ with $g^2=1$ and $w(g)=-1$.
	Since
	\[\Gamma(\pi_2 X)\cong \Gamma(\Z\pi)\oplus \Gamma(M)\oplus \Z\pi\otimes_\Z M,\]
	$\Tors(\Z^w\otimes_{\Z\pi}\Gamma(\Z\pi))$ is a summand of $\Tors(\Z^w\otimes_{\Z\pi}\Gamma(\pi_2X))$.
\end{proof}
\begin{prop}
	\label{prop:rp4cp2}
	There is a finite Poincar\'e complex with the same quadratic $2$-type as $\RP^4\#\CP^2$ which is not homotopy equivalent to $\RP^4\#\CP^2$.
\end{prop}

\begin{proof}	
	Let $B$ be the Postnikov $2$-type of $\RP^4\#\CP^2$. By \cref{lem:2.1.1}, $B$-polarized Poincar\'e complexes up to homotopy equivalence over $B$ are classified by the image of their fundamental class in $H_4(B;\Z^w)$. Hence to understand them up to homotopy equivalence, we have to study the action of all self-homotopy equivalences of $B$ on $H_4(B;\Z^w)$. As $\pi_2(B)$ is free, the $k$-invariant of $B$ is trivial and the map $B\to B\Z/2$ admits a splitting $s$. Using the Serre spectral sequence and the proof of \cref{prop:2.2.2ii}, one computes	
	
	\begin{align*}
	H_4(B;\Z^w)&\cong\Z^w\otimes_{\Z[\Z/2]}H_4(\wt B;\Z)\oplus H_4(B\Z/2;\Z^w)\\
	&\cong \Z^w\otimes_{\Z[\Z/2]}\Gamma(\Z[\Z/2])\oplus\Z/2\cong \Z\oplus\Z/2\oplus\Z/2.
	\end{align*}
	
	By obstruction theory, for every self-homotopy equivalence $\phi$ of $B$ we have $\phi\circ s\simeq s$. Hence $\phi$ acts diagonally on $H_4(B;\Z^w)$ with the identity on $H_4(B\Z/2;\Z^w)$ and on $\Z^w\otimes_{\Z[\Z/2]}H_4(\wt B;\Z)$ by the map on $\pi_2(B)$ induced by $\phi$. The automorphisms of $\pi_2(B)\cong \Z[\Z/2]$ are given by multiplicaton with $\{\pm 1,\pm g\}$ where $g\in\Z/2$ is the generator. It follows that the action of $\phi$ on $H_4(B;\Z^w)$ is given by $\pm 1$ and thus is the identity on the torsion subgroup. In particular, it is the identity on the torsion subgroup $\Z/2$ of $\Z^w\otimes_{\Z[\Z/2]} \Gamma(\Z[\Z/2]))$. It thus follows from \cref{thm:teichner} that there is precisely one Poincar\'e complex which has the same quadratic $2$-type as $\RP^4\#\CP^2$ but is not homotopy equivalent to it.
\end{proof}
\bibliographystyle{amsalpha}
\bibliography{gamma}

\providecommand{\bysame}{\leavevmode\hbox to3em{\hrulefill}\thinspace}
\providecommand{\MR}{\relax\ifhmode\unskip\space\fi MR }
\providecommand{\MRhref}[2]{%
  \href{http://www.ams.org/mathscinet-getitem?mr=#1}{#2}
}
\providecommand{\href}[2]{#2}
\begin{thebibliography}{KKR92}

\bibitem[Ada78]{adams}
John~Frank Adams, \emph{Infinite loop spaces}, Annals of Mathematics Studies,
  vol.~90, Princeton University Press, Princeton, N.J.; University of Tokyo
  Press, Tokyo, 1978. \MR{505692}

\bibitem[Bau88]{bauer}
Stefan Bauer, \emph{The homotopy type of a {$4$}-manifold with finite
  fundamental group}, Algebraic topology and transformation groups
  ({G}\"ottingen, 1987), Lecture Notes in Math., vol. 1361, Springer, Berlin,
  1988, pp.~1--6. \MR{979503}

\bibitem[Bre72]{bredon-trans}
Glen~E. Bredon, \emph{Introduction to compact transformation groups}, Academic
  Press, New York-London, 1972, Pure and Applied Mathematics, Vol. 46.
  \MR{0413144}

\bibitem[HK88]{hambleton-kreck}
Ian Hambleton and Matthias Kreck, \emph{On the classification of topological
  {$4$}-manifolds with finite fundamental group}, Math. Ann. \textbf{280}
  (1988), no.~1, 85--104. \MR{928299}

\bibitem[HKT94]{HKT94}
Ian Hambleton, Matthias Kreck, and Peter Teichner, \emph{Nonorientable
  {$4$}-manifolds with fundamental group of order {$2$}}, Trans. Amer. Math.
  Soc. \textbf{344} (1994), no.~2, 649--665. \MR{1234481}

\bibitem[HM78]{HM78}
I.~{Hambleton} and R.~J. {Milgram}, \emph{{Poincar\'e transversality for double
  covers.}}, {Can. J. Math.} \textbf{30} (1978), 1319--1330 (English).

\bibitem[KKR92]{KKR92}
Myung~Ho Kim, Sadayoshi Kojima, and Frank Raymond, \emph{Homotopy invariants of
  nonorientable {$4$}-manifolds}, Trans. Amer. Math. Soc. \textbf{333} (1992),
  no.~1, 71--81. \MR{1028758}

\bibitem[KPT]{KPT}
Daniel Kasprowski, Mark Powell, and Peter Teichner, \emph{Four-manifolds up to
  connected sum with complex projective planes}, arXiv:1802.09811.

\bibitem[Kre99]{kreck}
Matthias Kreck, \emph{Surgery and duality}, Ann. of Math. (2) \textbf{149}
  (1999), no.~3, 707--754. \MR{1709301}

\bibitem[Tei92]{teichnerthesis}
Peter Teichner, \emph{Topological 4-manifolds with finite fundamental group},
  Ph.D. thesis, University of Mainz, Germany, 1992, Shaker Verlag, ISBN
  3-86111-182-9.

\bibitem[Whi50]{whitehead}
J.~H.~C. Whitehead, \emph{A certain exact sequence}, Ann. of Math. (2)
  \textbf{52} (1950), 51--110. \MR{0035997}

\end{thebibliography}
\end{document}